\documentclass[11pt]{article}
\usepackage[utf8]{inputenc}
\usepackage[T1]{fontenc}
\usepackage{comment}
\usepackage{authblk}
\usepackage{centernot}
\usepackage{relsize}
\usepackage{changepage}
\usepackage{mathtools}
\usepackage[english]{babel}
\usepackage{amsmath}
\usepackage{amsthm}
\usepackage{amsfonts}
\usepackage{amssymb}
\usepackage{graphicx}
\usepackage[usenames,dvipsnames]{xcolor}
\theoremstyle{plain}
\newtheorem{theorem}{Theorem}[section]

\newtheorem{lemma}[theorem]{Lemma}

\theoremstyle{definition}

\newtheorem{remark}[theorem]{Remark}

\usepackage{empheq}

\makeatother
\definecolor{lyucol}{rgb}{0.5, 0.1, 0.1}
\definecolor{gray}{rgb}{0.5, 0.5, 0.5}

\usepackage[left=2cm,right=2cm,top=2cm,bottom=2cm]{geometry}
\usepackage{amssymb}
\usepackage{hyperref}
\makeatletter

\newlength{\dhatheight}

\usepackage[normalem]{ulem}
\usepackage{dsfont}
\usepackage{accents}
\usepackage{verbatim}
\usepackage{ulem}
\usepackage{pgf,tikz,pgfplots}
\pgfplotsset{compat = 1.17}
\usepackage[all]{xy}

\newcommand{\eps}{\varepsilon}

\newcommand{\e}{\mathrm{e}}
\newcommand{\Po}{\mathrm{Po}}

\renewcommand{\Pr}{\mathbb{P}}

\newcommand{\cC}{\mathcal{C}}

\newcommand{\cE}{\mathcal{E}}

\newcommand{\cO}{\mathcal{O}}

\newcommand{\ft}{\mathfrak{t}}

\usepackage{subfig}

\usepackage{enumerate}

\title{\scshape
Color-avoiding percolation of random graphs: between the subcritical and the intermediate regime}

\author{Lyuben Lichev}
\affil{Institut Camille Jordan, University Jean Monnet, Saint-Etienne, France}

\begin{document}

\maketitle
 
\begin{abstract}
Fix a graph $G$ in which every edge is colored in some of $k\ge 2$ colors. Two vertices $u$ and $v$ are CA-connected if $u$ and $v$ may be connected using any subset of $k - 1$ colors. CA-connectivity is an equivalence relation dividing the vertex set into classes called CA-components.

In two recent papers, R\'ath, Varga, Fekete, and Molontay, and Lichev and Schapira studied the size of the largest CA-component in a randomly colored random graph. The second of these works distinguished and studied three regimes (supercritical, intermediate, and subcritical) in which the largest CA-component has respectively linear, logarithmic, and bounded size. In this short note, we describe the phase transition between the intermediate and the subcritical regime.
\end{abstract}
\noindent
MSC Class: 05C80, 60C05, 60K35, 82B43

\section{Introduction}

The model of color-avoiding percolation (or CA-percolation, for short) is defined as follows. Every edge in a graph $G$ is colored in some of $k\ge 2$ colors. Two vertices $u$ and $v$ in $G$ are said to be \emph{CA-connected} if $u$ and $v$ may be connected using any subset of $k - 1$ colors.  CA-connectivity defines an equivalence relation on the vertex set of $G$ whose classes are called \emph{CA-components}.

CA-percolation was introduced by Krause, Danziger, and Zlati\'c~\cite{KDZ16, KDZ17}. They studied vertex-colored graphs and analyzed a vertex analog of CA-connectivity. While some empirical observations were made for scale-free networks, the focus was put on Erd\H{o}s-R\'enyi random graphs due to their better CA-connectivity~\cite{KDZ16}. In a subsequent work, Kadovi\'c, Krause, Caldarelli, and Zlati\'c~\cite{KKCZ18} defined mixed CA-percolation where both vertices and edges have colors. To a large extent, each of these foundational papers based their conclusions on experimental evidence.

A rigorous mathematical treatment of the subject was initiated by R\'ath, Varga, Fekete, and Molontay~\cite{RVFM22}. They showed that under a certain subcriticality assumption, the size of the largest CA-component renormalized by $n$ converges in probability to a fixed constant, and studied the behavior of that constant close to criticality. Recently, the author and Schapira~\cite{LS22} showed that one may naturally distinguish a supercritical, an intermediate, and a subcritical regime, in which the size of the largest CA-component is respectively linear, logarithmic, and bounded. In this short note, we give a precise description of the size of the largest CA-component between the subcritical and the intermediate regime.

\subsection{Notation and terminology}

We reuse the notation of~\cite{LS22}. Namely, for a positive integer $m$, we denote $[m] = \{1,\ldots,m\}$. We reserve the notation $[k]$ to denote the set of colors, and the notation $[n]$ for the set of vertices of our graphs. 
Recall that for $p\in [0,1]$, the Erd\H{o}s-R\'enyi random graph $G(n,p)$ with parameters $n$ and $p$, or ER random graph for short, is the graph on the vertex set $[n]$ where the edge between any two distinct vertices is present with probability $p$, independently from all other edges.   

Consider now a non-increasing sequence of positive real numbers $\lambda_1\ge \dots \ge \lambda_k$, and define a family of $k$ independent Erd\H{o}s-R\'enyi random graphs $G_i= G(n,\tfrac{\lambda_i}{n})$ for $i\in [k]$ on the same vertex set $[n]$ (alternatively, this family can be seen as a multigraph $([n],E_1,\dots,E_k)$ where $E_i$ is the edge set of $G_i$). In order to easily refer to and distinguish the graphs $(G_i)_{i=1}^k$, we say that for every $i\in [k]$, the edges of $G_i$ are given color $i$. Define further
$$\Lambda = \lambda_1 + \dots +\lambda_k,\quad \text{and}\quad
\lambda_i^* = \Lambda - \lambda_i \quad \text{for every }i\in[k].$$
In particular, $\lambda_1^* \le \lambda_2^*\le \dots \le \lambda_k^*$. Also, we set 
$$G = G_1\cup \ldots\cup G_k,$$ 
and for $I\subseteq [k]$, 
$$G_I= \bigcup_{i\in I} G_i, \quad \text{and}\quad G^I = \bigcup_{i\in [k]\setminus I} G_i,$$
with the shorthand notation $G_i$ and $G^i$, respectively, when $I=\{i\}$.

Recall that two vertices $u$ and $v$ in $G$ are CA-connected if $u$ and $v$ are connected in each of the graphs $(G^{i})_{i=1}^k$. Moreover, CA-connectivity is an equivalence relation with classes called CA-components. The CA-component of a vertex $u\in [n]$ is denoted by $\widetilde \cC(u)$, and $|\widetilde \cC(u)|$ denotes its \emph{size}, that is, the number of vertices it contains. 

Finally, a sequence of events $(\cE_n)_{n\ge 1}$ is said to hold a.a.s.\ if $\mathbb P(\cE_n)\to 1$ as $n\to \infty$. We use standard asymptotic notations. Namely, given two positive real sequences $(f_n)_{n\ge 1}$ and $(g_n)_{n\ge 1}$, we write $f_n=o(g_n)$, or $f_n\ll g_n$, or $g_n\gg f_n$ if $f_n/g_n\to 0$ when $n\to\infty$. Also, we write $f_n=\mathcal O(g_n)$ if there exists a constant $C>0$ such that $f_n\le Cg_n$ for all $n\ge 1$, and $f_n = \Theta(g_n)$ if $f_n = \cO(g_n)$ and $g_n = \cO(f_n)$. Furthermore, we write $\xrightarrow{d}$ to denote convergence in distribution of a sequence of random variables, and $\xrightarrow{\mathbb P}$ for convergence in probability.

\subsection{Main result} 

The following theorem from~\cite{LS22} distinguishes a supercritical, an intermediate and a subcritical regime, as mentioned in the introduction.

\begin{theorem}[Theorem~1.1 in~\cite{LS22}]\label{thm:k>2}
Suppose that $k\ge 2$.
\begin{enumerate}[(i)]
    \item\label{pt i} There exists $a_1\in [0,1)$  such that
    \[\frac{\max_{u\in [n]}|\widetilde \cC(u)|}{n}\xrightarrow[n\to \infty]{\Pr} a_1.\]
    Moreover, one has $a_1>0$ if and only if  $\lambda_1^* > 1$.
    \item\label{pt ii} If $\lambda_k^* > 1 > \lambda_{k-1}^*$, then there is a positive constant $a_2$ such that
    \[\frac{\max_{u\in [n]}|\widetilde \cC(u)|}{\log n}\xrightarrow[n\to \infty]{\Pr} a_2.\]
    \item\label{pt iii} If $ \lambda_k^*<1$, then a.a.s.\ $\max_{u\in [n]}|\widetilde \cC(u)| \le k$. Moreover, there are positive constants $\beta_2, \ldots, \beta_k$, such that
    \[(N_2, \ldots, N_k)\xrightarrow[n\to \infty]{d} \bigotimes_{\ell=2}^k  \Po(\beta_\ell),\]
    that is, the random variables $(N_\ell)_{\ell=2}^k$ jointly converge in distribution to $k-1$ independent Poisson variables with means $\beta_2, \ldots, \beta_k$ as $n\to \infty$.
\end{enumerate}
\end{theorem}

Our main result fills the gap between the last two parts of Theorem~\ref{thm:k>2}.

\begin{theorem}\label{thm:critical}
Fix a function $\zeta = \zeta(n)$ tending to $0$ as $n\to \infty$. Suppose that $\lambda_k^* = 1+\zeta$ and $\lambda_{k-1}^* < 1$ is a fixed constant. Then,
\begin{enumerate}[(i)]
    \item\label{crit i} if $\tfrac{\log \zeta^{-1}}{\log n}\to 0$ as $n\to \infty$, then 
    \[ \frac{\log(\zeta^{-1})\max_{u\in [n]} |\widetilde\cC(u)|}{\log n} \xrightarrow[n\to \infty]{\Pr} 1.\]
    \item\label{crit ii} if there is $\eps > 0$ such that $\zeta\le n^{-\eps}$ for all sufficiently large $n$, then
\begin{equation}\label{eq:crit 2} 
\sup_{n\ge 1}\; \Pr(\max_{u\in [n]} |\widetilde\cC(u)|\ge M) \xrightarrow[M\to \infty]{} 0.
\end{equation}
\end{enumerate}
\end{theorem}
\noindent
Note that Proposition 1.3 from~\cite{LS22} shows~\eqref{eq:crit 2} when $\zeta = 0$. In fact, the proof of the second point of Theorem~\ref{thm:critical} is inspired by the one of the above proposition, and is based on a combination of a first moment argument and a careful analysis of the connected components of $(G_i)_{i=1}^k$. The proof of the first part combines very precise though by now standard results on the component structure of a subcritical ER random graph with a non-trivial optimization step. 

\paragraph{Plan of the paper.} In Section~\ref{sec:prelims}, we present some preliminary results. While Sections~\ref{sec:stoc.dom} and~\ref{sec:proof.ii} mention already known results whose proves we omit, Section~\ref{sec.prel.intersection} gathers several technical lemmas that constitute the optimization step in the proof of Theorem~\ref{thm:critical}. Section~\ref{sec 3} is dedicated to the proof itself. We finish with a short discussion in Section~\ref{sec 4}.

\section{Preliminaries}\label{sec:prelims}

To begin with, we recall a version of the well-known Chernoff's inequality, see e.g. Corollary~2.3 in~\cite{JLR11}.

\begin{lemma}\label{lem:Chernoff}
Let\/ $X$ be a Binomial random variable with mean $\mu$.
Then, for all\/ $\delta\in (0,1)$, we have that 
\[\mathbb{P}[|X-\mu|\geq\delta\mu]\leq2\,\e^{-\delta^2\mu/3}.\]
\end{lemma}

\subsection{Barely supercritical regime: stochastic domination of the giant component}\label{sec:stoc.dom}
Fix $\lambda = \lambda(n)$ such that 
$n^{-1/3}\ll \lambda-1\ll 1$. Then, it is well-known that a.a.s.\ the graph $G(n,\tfrac \lambda n)$ has a unique largest connected component and
\begin{equation}\label{LLN.giant}
\frac{\max_{u\in [n]} |\cC(u)|}{(\lambda-1)n} \xrightarrow[n\to\infty]{\mathbb P} 2,
\end{equation}
see~\cite{Bol84, Luc90}. (Note that although both papers state most of their results for random graphs with a fixed number of edges, the corresponding results for $G(n,p)$ are derived analogously, as mentioned in Section~2 in~\cite{Bol84}.)

A version of the following lemma for $\lambda > 1$ appears as Lemma~2.10 in~\cite{LS22}. It provides a stochastic comparison between the size of the largest component in $G(n, \tfrac \lambda n)$ and binomial random subsets of $[n]$. The proof is a direct consequence of Chernoff's inequality for binomial random variables and~\eqref{LLN.giant}, and is therefore omitted.

\begin{lemma}\label{lem:stoch.giant}
Fix $\lambda = \lambda(n)$ such that $n^{-1/3}\ll \lambda-1\ll 1$ and $\varepsilon\in (0,2)$. Let $\cC_{\max}$ be the a.a.s.\ unique largest connected component of $G(n,\tfrac \lambda n)$. Let also $(X_v)_{v\in [n]}$ and $(Y_v)_{v\in [n]}$ be two sequences  
of i.i.d.\ Bernoulli random variables with respective parameters $(2-\eps)(\lambda-1)$ and $(2+\eps)(\lambda-1)$. Then, there is a coupling of these two sequences with $G(n,\tfrac{\lambda}n)$ such that a.a.s.\ one has
\begin{equation}\label{eq:inclusion.giant}
\{v : X_v=1\} \subseteq \cC_{\max} \subseteq \{v : Y_v=1\}.
\end{equation} 
\end{lemma}

\subsection{On the intersection of a giant with independent subcritical clusters}\label{sec.prel.intersection}

For every $t>0$, define $I_t = t - 1- \log t$. It is easy to check that $I_t$ is a positive constant for every positive $t$ except 1. 

Now, fix $\lambda\in (0,1)$. For every $s\ge 1$, we denote by $\ft_s$ the number of trees of size $s$ in the random graph $G(n,\tfrac{\lambda}{n})$. The next lemma gathers our knowledge for $(\ft_s)_{s\ge 1}$. Although it may be deduced by patching several already known results, we provide a sketch of proof for coherence and completeness.

\begin{lemma}\label{lem:components}
Fix $\lambda\in (0,1)$, a function $\omega = \omega(n)\to \infty$ as $n\to \infty$ such that $\omega(n) = o(\log\log n)$, and $\ell = \ell(\lambda, n) = I_{\lambda}^{-1} \left(\log n - \tfrac{5}{2}\log\log n\right) - \omega$. Then, a.a.s.\ for every $s\le \ell$, $\mathfrak{t}_s = (1+o(1))\mathbb E \mathfrak{t}_s\gg 1$. Moreover, a.a.s.\ there are at most $\e^{2I_{\lambda}\omega}$ trees of size between $\ell$ and $\ell+2\omega$ in $G(n,\tfrac{\lambda}{n})$. Finally, a.a.s.\ there are no trees of size more than $\ell+2\omega$.
\end{lemma}
\begin{proof}[Sketch of proof]
To begin with, we recall that for every $s\ge 1$, there are $s^{s-2}$ trees on $s$ vertices (statement known under the name Cayley's formula). Thus, for every positive integer $s\le 2\ell$,
\begin{equation}\label{eq:Et}
\mathbb E \mathfrak{t}_s = \binom{n}{s} s^{s-2} \left(\frac{\lambda}{n}\right)^{s-1}\left(1-\frac{\lambda}{n}\right)^{s(n-s)+(s-1)(s-2)/2} =  \cO\left(\frac{n\cdot \e^{-I_{\lambda}\cdot s}}{\lambda \sqrt{2\pi s^5}}\right),
\end{equation}
where the second equality comes from an application of Stirling's formula for $s!$ and the fact that $s^2\ll n$. Note that if additionally $s\gg 1$, 
\begin{equation}\label{eq:s>>1}
\mathbb E \mathfrak{t}_s = (1+o(1))\frac{n\cdot \e^{-I_{\lambda}\cdot s}}{\lambda \sqrt{2\pi s^5}}.
\end{equation}

In particular, for every $s\in [\ell]$, $\mathbb E \ft_s\ge \mathbb E\ft_{\ell}\gg 1$. Moreover, the second moment of $\mathfrak{t}_s$ for $s\in [\ell]$ is given by
\begin{align*}
\mathbb E \mathfrak{t}_s^2 
&=\; \mathbb E \mathfrak{t}_s + \binom{n}{s}\binom{n-s}{s} s^{2(s-2)} \left(\frac{\lambda}{n}\right)^{2(s-1)}\left(1-\frac{\lambda}{n}\right)^{2s(n-2s)+s^2+(s-1)(s-2)}\\
&=\; \mathbb E \mathfrak{t}_s + (\mathbb E \mathfrak{t}_s)^2 \left(1-\frac{\lambda}{n}\right)^{-s^2} \left(\prod_{i=0}^{s-1} \frac{n-s-i}{n-i}\right) = (\mathbb E \mathfrak{t}_s)^2 + \left(\mathbb E \mathfrak{t}_s + \cO\left(\frac{(s\mathbb E \mathfrak{t}_s)^2}{n}\right)\right).
\end{align*}

Thus, for every $s\in [\ell]$, Chebyshev's inequality implies that
\begin{equation*}
\Pr(|\mathfrak{t}_s - \mathbb E \mathfrak{t}_s|\ge (\mathbb E \mathfrak{t}_s)^{2/3})\le \frac{\mathrm{Var}(\mathfrak{t}_s)}{(\mathbb E \mathfrak{t}_s)^{4/3}} = \frac{1}{(\mathbb E \mathfrak{t}_s)^{1/3}} + \cO\left(\frac{s^2(\mathbb E \mathfrak{t}_s)^{2/3}}{n}\right) = \frac{1}{(\mathbb E \mathfrak{t}_s)^{1/3}} + \cO\left( \frac{1}{n^{1/4}}\right),
\end{equation*}
which together with a union bound over all $s\in [\ell]$ and the fact that $\sum_{s=1}^{\ell} (\mathbb E \mathfrak{t}_s)^{-1/3} = \cO((\mathbb E \mathfrak{t}_{\ell})^{-1/3}) = o(1)$ shows that a.a.s.\ for every $s\le \ell$, $|\mathfrak{t}_s - \mathbb E \mathfrak{t}_s|\ge (\mathbb E \mathfrak{t}_s)^{2/3}$. Using that for every $s\in [\ell]$ we have $\mathbb E \mathfrak{t}_s\ge \mathbb E \mathfrak{t}_{\ell}\gg 1$ shows that in particular, a.a.s.\ $\mathfrak{t}_s = (1+o(1))\mathbb E \mathfrak{t}_s$ for every $s\in [\ell]$.

Moreover, combining Markov's inequality for the sum $\sum_{s=\ell}^{\ell+2\omega} \ft_s$ with the fact that for every $s\in [\ell, \ell+2\omega]$, $\mathbb E \ft_s\le \mathbb E \ft_\ell = \cO(\e^{I_{\lambda}\omega})$, shows that there are a.a.s.\ at most $\e^{2I_{\lambda}\omega}$ trees of size between $\ell$ and $\ell+2\omega$. Finally, the last point follows by combining the fact that there are a.a.s.\ no connected components in $G(n, \tfrac \lambda n)$ of size more that $2\ell$ (see e.g. Theorem~4.4 in~\cite{Hof16}),~\eqref{eq:s>>1} and Markov's inequality for the sum of $\ft_s$ over the interval $s\in [\ell+2\omega, 2\ell]$.
\end{proof}

\begin{remark}\label{rem:tk}
Recall that for $\lambda\in (0,1)$ and for every function $\nu = \nu(n)\to \infty$ as $n\to \infty$, a.a.s.\ there are only $\cO(\nu)$ vertices in components in $G(n,\tfrac {\lambda} n)$ that contain a cycle (see e.g.\ Lemma~2.11 from~\cite{FK16}). In particular, the number of components containing a cycle is a.a.s.\ negligible to $\ft_s$ for every $s\in [\ell]$. Thus, to avoid clutter, we abuse notation and identify $\ft_s$ and the number of components of size $s$ for every $s\ge 1$.
\end{remark}

Now, we fix $q = q(n)\le 1$ and consider a sequence $(X_v)_{v \in [n]}$ of independent Bernoulli random variables with parameter $q$, and a graph $G(n,\tfrac{\lambda}{n})$ independent from them. We call a vertex $v\in [n]$ \emph{black} if $X_v = 1$, and \emph{white} otherwise. The aim of the following lemma is to show that the components of $G(n,\tfrac{\lambda}{n})$ whose size is very close to the size of the largest component do not typically contain abnormally many black vertices.

\begin{lemma}\label{lem:aux}
Fix any $\lambda\in (0,1)$ and $q = q(n)\to 0$ such that $\tfrac{\log q^{-1}}{\log n}\to 0$ as $n\to \infty$. Also, define $M_0 = \tfrac{\log n}{2\log q^{-1}}$. Then, with the notation of Lemma~\ref{lem:components}, a.a.s.\ every component of $G(n,\tfrac {\lambda} n)$ of size at least $\ell$ contains at most $M_0$ black vertices. 
\end{lemma}
\begin{proof}
Firstly, suppose that $q\ell\le (\log n)^{-1}$. By Lemma~\ref{lem:components} a.a.s.\ there are $\cO(\e^{2I_{\lambda}\omega}) = o(\log n)$ vertices in components of size between $\ell$ and $\ell+2\omega$, and none of size more than $\ell+2\omega$. We condition on this event. Then, by Markov's inequality the probability that a vertex in some of these components is black is at most $\cO(q\ell\e^{2I_{\lambda}\omega})= o(1)$. Thus, when $q\ell\le (\log n)^{-1}$, the statement is trivially satisfied.

On the other hand, suppose that $q\ell\ge (\log n)^{-1}$. Then, if $q\ell\le (\log n)^{1/3}$, Markov's inequality applied for any component $\cC$ in $G$ of size between $\ell$ and $\ell+2\omega$ gives that
\begin{equation*}
\Pr\left(\sum_{v\in \cC} X_v \ge \sqrt{\log n}\right) = \cO((\log n)^{-1/6}).
\end{equation*}
Thus, a union bound over all at most $\cO(\e^{2I_{\lambda}\omega}) = o((\log n)^{1/6})$ such components and the fact that $M_0 \ge \tfrac{\log n}{8\log\log n} \gg \sqrt{\log n}$ show the lemma in this case. 

Finally, if $q\ell\ge (\log n)^{1/3}$, Chernoff's inequality implies that
\begin{equation*}
\Pr\left(\sum_{v\in \cC} X_v \ge (1+\eps)q\ell\right) \le \exp\left(-(1+o(1))\frac{\eps^2 q \ell}{2}\right) = \cO((\log n)^{-1}).
\end{equation*}
However, using that $\tfrac{1}{\log q^{-1}}\gg q$ shows that $M_0\gg q\ell$ in this case. Thus, once again, a union bound over all $\cO(\e^{2I_{\lambda}\omega}) = o(\log n)$ components of size between $\ell$ and $\ell+2\omega$ finishes the proof.
\end{proof}

For every $s\ge 1$, we denote by $\widetilde Z_s$ the number of connected components $\cC$ in $G(n,\tfrac {\lambda} n)$ such that $\sum_{v\in \cC} X_v\ge s$. The following lemma characterizes the maximal $s$ for which $\widetilde Z_s\ge 1$. The idea is that by Lemma~\ref{lem:aux}, a.a.s.\ this maximum is attained by a component of size at most $\ell$. Then, our precise knowledge of the number of components of size $s\in [\ell]$ ensured by Lemma~\ref{lem:components} (and Remark~\ref{rem:tk}) allows us to conclude by an elementary computation.

\begin{lemma}\label{lem:main}
Fix $\lambda\in (0,1)$ and $q = q(n)\le 1$ such that $\tfrac{\log q^{-1}}{\log n}\to 0$ as $n\to \infty$. Then,
\begin{equation*}
\frac{(\log q^{-1}) \max\{s: \widetilde Z_s\ge 1\}}{\log n} \xrightarrow[n\to \infty]{\Pr} 1.
\end{equation*}
\end{lemma}
\begin{proof}
To begin with, we condition on the events in Lemma~\ref{lem:components} and on the sequence $(\ft_s)_{s\ge 1}$ defined therein. Recall that by Remark~\ref{rem:tk}, $\ft_s$ is identified with the number of components in $G(n, \tfrac \lambda n)$ of size $s$.

For every $s\ge 1$, define $\widetilde Z_s^{\ell}$ as the number of components $\cC$ of size at most $\ell$ such that $\sum_{v\in \cC} X_v\ge s$. By Lemma~\ref{lem:aux} a.a.s.\ $\widetilde Z_s^{\ell} = \widetilde Z_s$ for all $s\ge M_0+1$, so we may concentrate our attention on components of size at most $\ell$.

Now, given $M\in [M_0+1, \ell]$, we have that
\begin{align}\label{eq: Z_s}
\Pr(\exists s\in [M, \ell], \widetilde Z_s^{\ell}\ge 1) = \prod_{\substack{\cC\in G(n,\lambda/n),\\ |\cC|\le \ell}} (1-\Pr(\mathrm{Bin}(|\cC|, q)\ge M)) = \prod_{s=M}^{\ell} \bigg(1-\Pr(\mathrm{Bin}(s, q)\ge M)\bigg)^{\ft_s},
\end{align}
where the first product is over the set of connected components of $G(n,\tfrac{\lambda}{n})$ of size at least $\ell$. Thus, on the one hand, the fact that a.a.s.\ every component (even the ones of size more than $\ell$) has less than $M$ black vertices shows that $\Pr(\mathrm{Bin}(s, q)\ge M) = o(1)$ for every $s\in [\ell]$. On the other hand, the fact that $M\ge M_0\ge 2\ell q$ shows that for every $s\in [\ell]$, $\Pr(\mathrm{Bin}(s, q)\ge M) = \cO(\Pr(\mathrm{Bin}(s, q) = M)) = \cO(\tbinom{s}{M} q^{M}(1-q)^{s-M})$. This allows us to simplify~\eqref{eq: Z_s} to
\begin{align*}
\Pr(\widetilde Z_s^{\ell}\ge 1) = \exp\left(-\Theta\left(\sum_{s=M}^{\ell} \binom{s}{M} q^{M} (1-q)^{s-M} \ft_s\right)\right).
\end{align*}

Hence, using~\eqref{eq:s>>1} and the fact that $M\gg 1$, for $s\in [M,\ell-1]$ we have that $\tfrac{\ft_{s+1}}{\ft_{s}} = \e^{-I_{\lambda}} + o(1)$. Therefore, by comparing consecutive terms in the above sum, one may conclude that the maximal term is given by some $s_1$ satisfying 
\[\frac{(s_1+1)(1-q)}{s_1+1-M} \e^{-I_{\lambda}} = 1 + o(1), \text{ or equivalently } s_1 = (1+o(1)) \frac{M}{1-(1-q)\e^{-I_{\lambda}}}.\]
Furthermore, from $s = 2s_1$ on, say, the terms of the above sum decrease exponentially, and therefore 
\begin{equation}\label{eq:bounds sum}
\binom{s_1}{M} q^{M} (1-q)^{s_1-M} \ft_{s_1} \le \sum_{s=M}^{\ell} \binom{s}{M} q^{M} (1-q)^{s-M} \ft_s = \cO\left(s_1 \binom{s_1}{M} q^{M} (1-q)^{s_1-M} \ft_{s_1}\right).
\end{equation}
Now, by~\eqref{eq:Et} one may notice that both the lower and the upper bound are of the form 
\[\e^{(\log q+\cO(1)) M} \ft_{s_1} = \e^{(\log q+\cO(1)) M} n.\] 
Thus, for every $\eps > 0$, if $M\ge (1+\eps)\tfrac{\log n}{\log q^{-1}}$, then a.a.s.\ $\widetilde Z_M^{\ell} = 0$, while if $M\le (1-\eps)\tfrac{\log n}{\log q^{-1}}$, a.a.s.\ $\widetilde Z_M^{\ell} \ge 1$, which finishes the proof of the lemma.
\end{proof}

\subsection{Further preliminaries on connectivity and CA-connectivity} \label{sec:proof.ii}

The two results in this section may be found as Lemma~3.2 and Lemma~3.3 in~\cite{LS22}.

For $I\subset [k]$, write $\lambda_I^* = \Lambda - \sum_{i\in I} \lambda_i$. Also, for two vertices $u$ and $v$ in a graph $H$, we write $u\xleftrightarrow{H\,}v$ for the event that $u$ and $v$ are connected by a path in $H$, and $u\stackrel{H}{\centernot\longleftrightarrow}v$ for the opposite event.

\begin{lemma}[Lemma~3.2 in~\cite{LS22}]\label{lem:Ijconnect}
For every $j\in [k]$ and $I\subseteq [k]\setminus \{j\}$ such that $\max(\lambda_I^*,\lambda_j^*)<1$, one has 
$$\lim_{M\to \infty} \mathbb P\big(\exists u\neq v: u\xleftrightarrow{G^I}v,\,  u\xleftrightarrow{G^j}v,\, u\stackrel{G^{I\cup\{j\}}}{\centernot\longleftrightarrow}v, \, \max(|\cC^I(u)|,|\cC^j(u)|)\ge M\big) = 0$$
uniformly in $n$.
\end{lemma}

We call a set \emph{connected} if its vertices belong to the same connected component.

\begin{lemma}[Lemma~3.3 in~\cite{LS22}]\label{lem:intersection}
Suppose that $\lambda_{k-1}^* < 1$. Then,
\[\lim_{M\to \infty} \Pr(\exists S\subseteq [n]: |S|\ge M, \text{$S$ is connected in each of $G^1,\dots,G^{k-1}$ but not in $G^k$}) = 0\]
uniformly in $n$.
\end{lemma}

\section{Proof of Theorem~\ref{thm:critical}}\label{sec 3}

We first show the second part of the theorem. Denote by $\widetilde \cC_{\max}$ the largest CA-component in $G$. As the vertices of $\widetilde \cC_{\max}$ form a connected set in $G^i$ for every $i\in [k-1]$, by Lemma~\ref{lem:intersection} we have that
\[\lim_{M\to \infty}\sup_{n\ge M}\,  \Pr(|\widetilde \cC_{\max}|\ge M, \text{$\widetilde \cC_{\max}$ is not connected in $G_k$}) = 0.\]
On the other hand, if $\widetilde \cC_{\max}$ is connected in $G_k$, it is obtained by intersecting a connected component in $G^k$ and one in $G_k$. Moreover, by~\eqref{LLN.giant} the largest component in $G(n, \lambda_k^*/n)$ is a.a.s.\ of size $(2+o(1)) \zeta n = \cO(n^{1-\eps})$, and moreover, a.a.s.\ the size of the largest connected component in the subcritical random graph $G_k$ is $\cO(\log n)$. We condition on these two events and reveal the sizes of the connected components of $G_k$ and $G^k$. 

Now, note that for every $s\ge 1$, the probability that an $s$-tuple of vertices is connected in both $G_k$ and $G^k$ may be expressed in terms of the component sizes only. In particular, a union bound over all components in $G^k$, all components in $G_k$ and all $(\lfloor 2\eps^{-1}\rfloor + 1)$-tuples of vertices in $G$ shows that the probability of having a set of $\lfloor 2\eps^{-1}\rfloor + 1$ vertices that are connected both in $G^k$ and in $G_k$ is bounded from above by
\begin{equation}\label{eq:fst moment}
n^2\cdot n^{\lfloor 2\eps^{-1}\rfloor + 1}\cdot \cO\bigg(\bigg(\frac{n^{1-\eps}}{n}\bigg)^{\lfloor 2\eps^{-1}\rfloor + 1}\bigg(\frac{\log n}{n}\bigg)^{\lfloor 2\eps^{-1}\rfloor + 1}\bigg) = o(1).
\end{equation}
Hence, using that $n\ge M$,
\begin{align*}
\lim_{M\to \infty}\sup_{n\ge M}\, \Pr(|\widetilde \cC_{\max}|\ge M) 
&\le\; \lim_{M\to \infty}\sup_{n\ge M}\, \Pr(|\widetilde \cC_{\max}|\ge M, \text{$\widetilde \cC_{\max}$ is not connected in $G_k$})\\
&+\;\lim_{n\to \infty}\, \Pr(|\widetilde \cC_{\max}|\ge \lfloor 2\eps^{-1}\rfloor + 1, \text{$\widetilde \cC_{\max}$ is connected in $G_k$}) = 0,
\end{align*}
which finishes the proof of the second part.

For the first part, the fact that a.a.s.\ every connected component in $G^k$ except for the largest one contains $n^{o(1)}$ vertices (see Theorem~11 in~\cite{Bol84}), and a computation analogous to~\eqref{eq:fst moment} show that a.a.s.\ for every such component in $G^k$ and every component in $G_k$, their common intersection is of size at most 3. On the other hand, Lemma~\ref{lem:intersection} implies that a.a.s.\ every CA-component with at least $M_0\gg 1$ vertices (with $M_0$ defined as in Lemma~\ref{lem:aux}) is obtained as intersection of the largest component in $G^k$ with some connected component in $G_k$. Hence, combining~\eqref{LLN.giant}, Lemma~\ref{lem:stoch.giant} for $\eps = 1$, and Lemma~\ref{lem:main} for Bernoulli random variables with parameters $q = \zeta$ and $q = 4\zeta$ finishes the proof since in both cases a.a.s.\ $\max\{s:\widetilde Z_s\ge 1\} = (1+o(1)) \tfrac{\log n}{\log \zeta^{-1}}$.

\section{Concluding remarks}\label{sec 4}
In this note, we extended a line of research dedicated to color-avoiding percolation in the context of ER random graphs. In particular, our Theorem~\ref{thm:critical} identifies the correct order of the size of the largest CA-component when $\lambda_k^* = 1+o(1)$, and provides the constant in front of the first term when $\log(\lambda_k^* - 1) = o(\log n)$.

Some intriguing questions remain open.
\begin{itemize}
\item As mentioned in~\cite{LS22}, the range where the intermediate regime is well understood is contained in the second part of Theorem~\ref{thm:k>2}. It would be interesting to identify the behavior of the largest CA-component when $\lambda_{k-1}^* \ge 1 > \lambda_1^*$.
\item While we know that under the conditions of Theorem~\ref{thm:k>2}~\eqref{crit ii}, the size of the largest CA-component is a tight random variable, we do not have good understanding of its distribution. Finding this distribution is another open problem deserving attention.
\end{itemize}

\paragraph{Acknowledgements.} The author is grateful to Bruno Schapira and to Dieter Mitsche for many useful discussions.

\bibliographystyle{plain}
\bibliography{Refs}

\end{document}